\documentclass[11pt]{article} 
\usepackage[francais]{babel}
\usepackage{multicol} 
\usepackage{newlfont}
\usepackage{amsmath,amssymb}

\def\ait{\mathbb{A}}    %{ \hbox{\bf {\it  A\hskip -10pt A}} }
\def\ppit{\mathbb{P}} %{ \hbox{\bf {\it I\hskip -2pt P}} }
\def\qit{\mathbb{Q}} %{ \hbox{\bf {\it l\hskip -7pt Q}}}
\def\cit{\mathbb{C}} %{ \hbox{ l\hskip -6pt C\/}}
\def\fit{\mathbb{F}}
\def\git{\mathbb{G}}

\newcommand{\pf}{{\bf Proof.~}}

\newcommand{\qed}{\hfill~~\mbox{$\Box$}}

\newenvironment{proof}{\smallskip \noindent \pf}{\qed \bigskip}

\newtheorem{theorem}{Theorem}[subsection]
\newtheorem{proposition}[theorem]{Proposition}
\newtheorem{definition}[theorem]{Definition}
\newtheorem{lemma}[theorem]{Lemma}
\newtheorem{corollary}[theorem]{Corollary}
\newtheorem{notation}[theorem]{Notation}
\newtheorem{remark}[theorem]{Remark}
\newtheorem{example}[theorem]{Example}

\begin{document}

\title{\bf A canonical Frobenius structure}
\author{\sc{ Antoine Douai} \\
Laboratoire J.A Dieudonn\'e, UMR 6621,\\ 
Universit\'e de Nice, Parc Valrose, F-06108 Nice Cedex 2 FRANCE.\\ 
E-mail : douai@unice.fr
}

\date{}

\maketitle

\begin{center}{\bf Introduction}
\end{center}

This paper is the last of a series devoted to the construction of Frobenius 
structures on the base of a deformation of a convenient and nondegenerate Laurent polynomial $f$, defined on the torus $U=(\cit^{*})^{n}$. The motivations and the general setting are given in [DoSa] and [DoSa2]. In [D] we have explained how one can construct, using a result of Hertling and Manin [HeMa], Frobenius structures
which are determined by a restricted set of data (the ''initial conditions'').  However, these initial conditions are not unique and, starting from $f$, it is {\em a priori} possible to construct several Frobenius structures. The goal of this paper is to compare them, in fact to show that they are all isomorphic: finally, to a convenient and nondegenerate Laurent polynomial we associate 
{\em a canonical} Frobenius structure. 

Let us precise the situation: let 
$$F:U\times\cit^{r}\rightarrow\cit$$
be the subdiagram deformation of $f$ defined by
$$F(u,x)=f(u)+\sum_{i=1}^{r}x_{i}g_{i}(u)$$
where the $g_{i}$'s are some Laurent polynomials (we put $x=(x_{1},\cdots ,x_{r})$ and $u=(u_{1},\cdots ,u_{n})$). Here, subdiagram means that the Laurent polynomials $g_{1},\cdots ,g_{r}$ are linear combinations of monomials $u_{1}^{a_{1}}\cdots u_{n}^{a_{n}}$ where $a=(a_{1},\cdots , a_{n})$ belongs to the interior of the Newton polyhedron of $f$. One can attach to $F$ a Frobenius type structure on $\ait^{r}$, that is a t-uple
$$\fit =(\ait^{r}, E, \bigtriangledown , R_{0}, R_{\infty},\Phi ,g)$$
where $E$ is a free $\cit [x]$-module, $\Phi$ a Higgs field, $\bigtriangledown$ a flat connection on $E$, $g$ a metric, $R_{0}$ and $R_{\infty}$ two endomorphisms of $E$, these different objects satisfying some natural compatibility relations.
This is the initial condition and it is obtained by solving the Birkhoff problem for the Brieskorn lattice of $F$. 
Once $\fit$ is fixed, and up to the existence of a pre-primitive and homogeneous form, that is a $\bigtriangledown$-flat
section $\omega$ of $E$ satisfying an injectivity condition (IC), a generation condition (GC) and a homogeneity condition
(H), one can equip, following Hertling and Manin [HeMa], $(\cit^{\mu},0)$ with a Frobenius structure ($\mu$ is the global Milnor number of $f$).  

In this paper, we will take for $\omega$ the class of the volume form
$$\frac{du_{1}}{u_{1}}\wedge\cdots\wedge\frac{du_{n}}{u_{n}}$$
in $E$, the reason being that $\omega$ is the $\bigtriangledown$-flat extension to $E$ of the canonical primitive form attached to $f$ by  [DoSa, 4.d]. Then, $\omega$ satisfies the condition (IC) at least if the $g_{i}$'s are $\cit$-linearly independent, in which case we will say that the subdiagram deformation $F$ is {\em injective}.
Condition (H) follows from the homogeneity of the canonical primitive form attached to $f$ by {\em loc. cit.}. Let us have a closer look 
at (GC): the point is that this condition will set the deformation $F$ and thus the initial data $\fit$. $\omega$ will satisfy (GC) if
any element of 
$A_{f}$, the Jacobi algebra of $f$,
can be written as the class of a polynomial in $g_{1},\cdots , g_{r}, f$ with coefficients in $\cit$. Of course, this will be true if 
any element of $A_{f}$
can be written as the class of a polynomial in $g_{1},\cdots , g_{r}$ with coefficients in $\cit$, in which case we 
will say that $(g_{1},\cdots ,g_{r})$ is a {\em lattice} in $A_{f}$,
or if any element of $A_{f}$
can be written as the class of a polynomial in $f$ with coefficients in $\cit$. The latter case occurs when
the multiplication by $f$ on $A_{f}$ is {\em regular}, in particular if the critical values of $f$ are all distinct. 
We focuse now on the former case: let $(g_{1},\cdots ,g_{r})$ be a lattice in $A_{f}$. Then $\omega$ is pre-primitive and homogeneous but the desired Frobenius structure will depend {\em a priori} on the lattice $(g_{1},\cdots ,g_{r})$:
two different lattices could give two distinct Frobenius manifolds. We show:\\

\noindent {\bf Theorem 1.} {\em Let $f$ be a convenient and nondegenerate Laurent polynomial, $\mu$ its global Milnor number. Assume that there exists a lattice $(g_{1},\cdots ,g_{r})$ in $A_{f}$. Then the construction of
Hertling and Manin equips $(\cit^{\mu},0)$ with a canonical Frobenius structure. Up to isomorphism, this Frobenius structure doesn't depend on the lattice $(g_{1},\cdots ,g_{r})$.}\\

\noindent Thus, if there exists a lattice in $A_{f}$, it makes sense to speak of {\em the} Frobenius structure attached to a convenient and nondegenerate Laurent polynomial. Theorem 1 includes also the regular case: if the multiplication by $f$ is regular and if there exists a lattice in $A_{f}$ it follows from the discussion above that there are at least two ways to construct Frobenius structures. They will be isomorphic.

Up to a slightly stronger generation condition, we can give a global counterpart of Theorem 1: let  
$$F(u,x)=f(u)+\sum_{i=1}^{r}x_{i}g_{i}(u)$$ 
be an injective subdiagram deformation of $f$, $A_{F}$ its Jacobi algebra, which is a $\cit [x]$-module of finite type. We will say that $\omega$ satsifies $(GC)^{gl}$ (for the deformation $F$) if $(g_{1},\cdots ,g_{r})$ is a lattice in $A_{F}$, that is if any element of $A_{F}$ can be written as (the class of) a polynomial in $g_{1},\cdots , g_{r}$ with coefficients in $\cit [x]$. Let $a\in\cit^{r}$ and $\rho_{a}$ be the map defined by $\rho_{a}(x,y)=(x+a,y)$
for $(x,y)\in\cit^{r}\times (\cit^{\mu -r},0)$.\\

\noindent {\bf Theorem 2.} 
{\em Let $a\in\cit^{r}$ and assume that $\omega$ satisfies $(GC)^{gl}$ for $F$. Then,\\
$1)$ the canonical Frobenius structure attached by Theorem 1 to the convenient and nondegenerate Laurent polynomial $F_{a}:=F(\ .,a)$ is isomorphic to the pull-back by $\rho_{a}$ of the one attached to $f$,\\
$2)$ for any injective and subdiagram deformation $G$ of $f$, the canonical Frobenius structure attached by Theorem 1 to the convenient and nondegenerate Laurent polynomial $G_{a}:=G(\ .,a)$ is isomorphic to the pull-back by $\rho_{a}$ of the one attached to $f$.}\\

\noindent In other words, the canonical Frobenius structure attached by Theorem 1 to $f$ determines the canonical Frobenius structure attached by Theorem 1 to $G_{a}$ for any injective subdiagram deformation $G$. Theorems 1 and 2 are detailed in section 6.

This paper is organized as follows: in section 1, we recall the basic facts about the Frobenius type structures and their deformations. In section 2, we explain the construction of Hertling and Manin. Then we apply all this to a geometric situation: we define the canonical Frobenius type structures attached to a subdiagram deformation of a convenient and nondegenerate Laurent polynomial (section 3) and the canonical pre-primitive form (section 4). In section 5 we study the existence of universal deformations of the canonical Frobenius type structure. We show in particular that one can define global universal deformations along the space of the subdiagram monomials. Last, section 6 is devoted to the proof of Theorems 1 and 2.\\

\noindent {\bf Acknowledgements.} I thank C. Sabbah for many helpful discussions.\\

\noindent {\em Notations.} In this paper we will put $U=(\cit^{*})^{n}$, $u=(u_{1},\cdots,u_{n})$, $x=(x_{1},\cdots ,x_{r})$,
$$K=\cit [u,u^{-1}]=\cit [u_{1},\cdots ,u_{n},u_{1}^{-1},\cdots ,u_{n}^{-1}]$$
and
$$\frac{du}{u}=\frac{du_{1}}{u_{1}}\wedge\cdots\wedge\frac{du_{n}}{u_{n}}.$$
If $f$ is a Laurent polynomial, $A_{f}$ will denote its Jacobi algebra
$$\frac{K}{(\frac{\partial f}{\partial u_{1}},\cdots ,\frac{\partial f}{\partial u_{n}})}.$$

\section{Frobenius type structure}

\subsection{Frobenius type structure on a complex analytic manifold}
Let $M$ be a complex analytic manifold. Let us be given a t-uple
$$(M, E, \bigtriangledown , R_{0}, R_{\infty},\Phi ,g)$$
\noindent where\\

$\bullet$ $E$ is a locally free ${\cal O}_{M}$-module,\\

$\bullet$ $R_{0}$ and $R_{\infty}$ are ${\cal O}_{M}$-linear endomorphisms of $E$,\\

$\bullet$ $\Phi :E\rightarrow \Omega^{1}_{M}\otimes E$ is an ${\cal O}_{M}$-linear map,\\

$\bullet$ $g$ is a  {\em metric} on $E$, {\em i.e} a ${\cal O}_{M}$-bilinear form, symmetric and nondegenerate,\\

$\bullet$ $\bigtriangledown$ is a connection on $E$.\\

\begin{definition} The t-uple 
$$(M, E, \bigtriangledown , R_{0}, R_{\infty},\Phi ,g)$$
is a Frobenius type structure on $M$ if the following relations are satisfied:

\begin{center} $\bigtriangledown^{2}=0$, $\bigtriangledown (R_{\infty})=0$, $\Phi\wedge\Phi =0$, $[R_{0},\Phi ]=0$,
\end{center}
\begin{center} $\bigtriangledown (\Phi )=0$, $\bigtriangledown (R_{0})+\Phi =[\Phi ,R_{\infty}]$, 
\end{center} 
\begin{center}
 $\bigtriangledown (g)=0$, $\Phi^{*}=\Phi$, $R_{0}^{*}=R_{0}$, $R_{\infty}+R_{\infty}^{*}=rId$
\end{center}
for a suitable constant $r$. $^{*}$ denotes the adjoint with respect to $g$.
\end{definition}

\noindent We will use systematically the following lemma, which is a direct consequence of the definition:

\begin{lemma} Let $$(M, E, \bigtriangledown , R_{0}, R_{\infty},\Phi ,g)$$
be a Frobenius type structure on $M$. Then:\\
$1)$ $\bigtriangledown$ is flat.\\
$2)$ Let $\varepsilon$ be a $\bigtriangledown$-flat basis of $E$, $C=\sum_{i}C^{(i)}dx_{i}$ (resp. $B_{0}$, $B_{\infty}$)
the matrix of $\Phi$ (resp. $R_{0}$, $R_{\infty}$) in this basis. One has, for all $i$ and for all $j$,
\begin{center} $\frac{\partial C^{(i)}}{\partial x_{j}}=\frac{\partial C^{(j)}}{\partial x_{i}}$,
\end{center}
\begin{center} $[C^{(i)},C^{(j)}]=0$, 
\end{center} 
\begin{center}
 $[B_{0},C^{(i)}]=0$,
\end{center}
\begin{center}
 $C^{(i)}+\frac{\partial B_{0}}{\partial x_{i}}=[B_{\infty},C^{(i)}]$
\end{center}
\begin{center} $C^{(i)*}=C^{(i)}$, $B_{0}^{*}=B_{0}$, $B_{\infty}+B_{\infty}^{*}=rI$
\end{center}
($I$ is the identity matrix). The matrix $B_{\infty}$ is constant.
\end{lemma}

\begin{remark} $1)$ If $M=\{point\}$, a Frobenius type structure on $M$ is a t-uple
$(E, R_{0}, R_{\infty},g)$ where $E$ is a finite dimensional $\cit$-vector space, $R_{0}$ and $R_{\infty}$ are endomorphisms of $E$, and $g$ is a bilinear, symmetric and nondegenerate form on $E$ such that
\begin{center}
$R_{0}^{*}=R_{0}$, $R_{\infty}+R_{\infty}^{*}=rId$ 
\end{center}
for a suitable constant $r\in\cit$.\\
$2)$ We will also consider Frobenius type structures on $\ait^{r}$ that is t-uples
$$(\ait^{r}, E, \bigtriangledown , R_{0}, R_{\infty},\Phi ,g)$$
where $E$ is a free $\cit [x]$-module. $\bigtriangledown$ , $R_{0}$, $R_{\infty}$, $\Phi$ and $g$ are defined as above (replace ${\cal O}_{M}$-linear by $\cit [x]$-linear) and satisfy the relations of Definition 1.1.1.
\end{remark}

\subsection{Construction of Frobenius type structures}
Let $\pi :\ppit^{1}\times M\rightarrow M$ be the projection, ${\cal E}:=\pi^{*}E$ and $\nabla$ the meromorphic connection on $E$ defined by 
$$\nabla =\pi^{*}\bigtriangledown +\frac{\pi^{*}\Phi }{\tau}-(\tau R_{0}+R_{\infty})\frac{d\tau }{\tau}$$
where $\tau$ is the coordinate on the chart centered at infinity. Then $\nabla$ is flat if and only if the t-uple
$$(M, E, \bigtriangledown , R_{0}, R_{\infty},\Phi )$$
is a Frobenius type structure on $M$ (without metric). 
Conversely, a trivial bundle ${\cal E}$ on $\ppit^{1}\times M$ equipped with a flat connection $\nabla$, with logarithmic poles along $\{\infty\}\times M$ and with poles of order 1 along $\{0\}\times M$, enables us to construct a Frobenius type structure (without metric)
$$(M, E, \bigtriangledown , R_{0}, R_{\infty},\Phi )$$
where $E:={\cal E}_{|\{0\}\times M}$ (see for instance [Sab, chapitre VII] for the details).
One can also get in this way a Frobenius type structure 
$$(M, E, \bigtriangledown , R_{0}, R_{\infty},\Phi ,g)$$
with metric (see [Sab, chapitre VI, 2.b]).
All Frobenius type structures that we will consider are constructed in this way.

\subsection{Deformations of Frobenius type structures}

Since one knows how to define the pull-back of a bundle equipped with a connection, one can define, using section 1.2, the pull-back of a Frobenius type structure: if 
$\varphi :N\rightarrow M$ where $M$ and $N$ are two complex analytic manifolds and if
${\cal F}$ is a Frobenius type structure on $M$ then $\varphi^{*}{\cal F}$ is a Frobenius type structure on $N$. 

\begin{definition}
$1)$ If $\varphi$ is a closed immersion, one says that ${\cal F}$ is a {\em deformation} of $\varphi^{*}{\cal F} $.\\
$2)$ Two deformations of a same Frobenius type structure are {\em isomorphic} if one comes from the other by a base change inducing an isomorphism on the corresponding tangent bundles.\\
$3)$ A deformation $\tilde{ {\cal F}}$ of 
a Frobenius type structure ${\cal F}$ on $M$ is {\em universal} if any other deformation of 
${\cal F}$ can be obtained from $\tilde{ {\cal F}}$ after a unique base change, inducing the identity on $M$.
\end{definition}

\noindent If it exists, a universal deformation is unique, up to isomorphism.

\section{Hertling and Manin's theorem. Construction of Frobenius manifolds}

Let
$${\cal F}=(M, E, \bigtriangledown , R_{0}, R_{\infty},\Phi ,g),$$
be a Frobenius type structure on $M$, which can be a punctual germ of a complex analytic manifold, a simply connected complex analytic manifold (the analytic case) or $\ait^{r}$ (the algebraic case). 

\subsection{Pre-primitive forms}

\subsubsection{The analytic case}

Suppose first that $M$ is a punctual germ of a complex analytic manifold.
Let $\omega$ be a $\bigtriangledown$-flat section of $E$. 

\begin{definition} The period map attached to $\omega$ is the map 
\begin{eqnarray}
\varphi_{\omega}: & \Theta_{M} & \rightarrow  E \\
                  & \xi        & \mapsto      -\Phi_{\xi}(\omega )
\end{eqnarray}
\end{definition}

 The period map $\varphi_{\omega}$ can be seen as a $\bigtriangledown$-flat differential form: in coordinates,
$$\varphi_{\omega}=-\sum_{i=1}^{r}\Phi_{\partial_{x_{i}}}(\omega )dx_{i}.$$
Assume moreover that $\omega =\varepsilon_{1}$ where 
$\varepsilon =(\varepsilon_{1},\cdots ,\varepsilon_{\mu})$ is a $\bigtriangledown$-flat basis of $E$.
With the notations of Lemma 1.1.2, one then gets 
$$\varphi_{\omega}=-\sum_{j=1}^{\mu}(\sum_{i=1}^{r}C^{(i)}_{j1}(x)dx_{i})\varepsilon_{j}.$$
Lemma 1.1.2 2) shows also that the differential form $\sum_{i=1}^{r}C^{(i)}_{j1}(x)dx_{i}$ is $d$-closed : 
let $\Gamma_{j1}$ be the function such that $\Gamma_{j1}(0)=0$ and $d\Gamma_{j1}(x)=\sum_{i=1}^{r}C^{(i)}_{j1}(x)dx_{i}.$
Define

\begin{eqnarray}
\chi_{\omega}^{\varepsilon}: & M & \rightarrow  E \\
                  & x & \mapsto   \sum_{j=1}^{\mu}\Gamma_{j1}(x)\varepsilon_{j}.  
\end{eqnarray}

\noindent The basis $\varepsilon$ being fixed, $\chi_{\omega}^{\varepsilon}$ can also be seen as a map

\begin{eqnarray}
\chi_{\omega}^{\varepsilon}: & M & \rightarrow  \cit^{\mu} \\
                  & x & \mapsto   (\Gamma_{11}(x),\cdots ,\Gamma_{\mu 1}(x)) 
\end{eqnarray}

\begin{definition} $\chi_{\omega}^{\varepsilon}$ is the primitive map attached to the  $\bigtriangledown$-flat section $\omega$ and to the basis $\varepsilon$.
\end{definition}

\begin{remark} Up to isomorphism, the map $\chi_{\omega}^{\varepsilon}$ doesn't depend on the basis $\varepsilon$. We will omit the index $\varepsilon$ : there will be no confusion because we will always work with M. Saito's canonical basis (see section 3.3).
\end{remark}

Let $m$ be the maximal ideal of ${\cal O}_{M}$. The index $^{o}$ will denote the operation ''modulo $m$''.

\begin{definition} Let $\omega$ be a $\bigtriangledown$-flat section of $E$. One says that $\omega$ is  {\em pre-primitive} if\\

$(GC)$ $\omega^{o}$ and its images under the iteration of the maps $R_{0}^{o}$ and $\Phi^{o}_{\xi}$ (for all $\xi$) generate $E^{o}$,\\

$(IC)$ $\varphi_{\omega}^{o}:\Theta_{M}^{o}\rightarrow E^{o}$ is injective.
\end{definition}

\begin{remark} 
$1)$ If $M=\{point\}$ the condition (IC) is empty. Assume moreover that $R_{0}$ is regular (i.e its characteristic polynomial is equal to its minimal polynomial): there exists $\omega$ such that 
$$\omega ,R_{0}(\omega ),\cdots ,R_{0}^{\mu -1}(\omega )$$
is a basis of $E$ over $\cit$. $\omega$ is thus pre-primitive.\\
$2)$ If (GC) is satisfied, it is also satisfied in the neighborhood of $0$: $E$ is then generated by $\omega$ and its images under iteration of the maps $R_{0}$ et $\Phi_{\xi}$ (for all $\xi$). 
\end{remark}

Let now $M$ be a simply connected complex analytic manifold. The period map attached to the $\bigtriangledown$-flat section $\omega$ is the ${\cal O}_{M}$-linear map is defined as in Definition 2.1.1. One defines also the primitive map $\chi_{\omega}^{\varepsilon}$, attached to the  $\bigtriangledown$-flat section $\omega$ and to the basis $\varepsilon$ : since $M$ is simply connected, $\chi_{\omega}^{\varepsilon}$ is holomorphic on $M$.
The definition of the pre-primitive forms depends now on the origin: if $a\in M$, 
$m^{a}$ will denote the maximal ideal of ${\cal O}_{M,a}$ and the index $^{a}$ the operation ''modulo $m^{a}$''.

\begin{definition} Let $\omega$ be a $\bigtriangledown$-flat section of $E$, $a\in M$. 
We will say that $\omega^{a}$ satisfies $(GC)$ if $\omega^{a}$ and its images under the iteration of the maps $R_{0}^{a}$ and $\Phi_{\xi}^{a}$ (for all $\xi$) generate $E^{a}$ and that $\omega^{a}$ satisfies $(IC)$ if
$$\varphi_{\omega}^{a}:\Theta_{M}^{a}\rightarrow E^{a}$$
is injective. One says that $\omega$ is  {\em pre-primitive} for the origin $a$ if $\omega^{a}$ satisfies $(GC)$ and $(IC)$.
\end{definition}

\subsubsection{The algebraic case} 

Let
$$\fit =(\ait^{r}, E, \bigtriangledown , R_{0}, R_{\infty},\Phi ,g)$$
be a Frobenius type structure on $\ait^{r}$.
The period map attached to $\omega$ is now a $\cit [x]$-linear map, defined on the 
Weyl algebra $\ait^{r}(\cit )=\cit [x]<\partial_{x}>$,

\begin{eqnarray}
\varphi_{\omega}: & \ait^{r}(\cit ) & \rightarrow  E \\
                  & \xi & \mapsto     -\Phi_{\xi}(\omega ) 
\end{eqnarray}

\noindent One defines also the primitive map $\chi_{\omega}^{\varepsilon}$, attached to the $\bigtriangledown$-flat section $\omega$ and to the basis $\varepsilon$. The index $^{a}$ will denote the operation ''modulo $(x-a)$''.

\begin{definition} Let $\omega$ be a $\bigtriangledown$-flat section of $E$.\\
$1)$ We will say that $\omega$  satisfies the condition $(GC)^{gl}$ if $\omega$ and its images under the iteration of the maps $R_{0}$ and 
$\Phi_{\xi}$ (for all $\xi$) generate the $\cit [x]$-module $E$
and that $\omega$ satisfies the condition $(IC)^{gl}$ if 
$$\varphi_{\omega}:\ait^{r}(\cit )\rightarrow E$$ 
is injective. We will say that $\omega$ is  {\em globally pre-primitive} if $\omega$ satisfies $(GC)^{gl}$ and $(IC)^{gl}$.\\
\noindent $2)$ Let $a\in\ait^{r}$. We will say that $\omega^{a}$ satisfies $(GC)$ if $\omega^{a}$ and its images under the iteration of the maps $R_{0}^{a}$ and $\Phi_{\xi}^{a}$ (for all $\xi$) generate $E^{a}$ and that $\omega^{a}$ satisfies $(IC)$ if $\varphi_{\omega}^{a}$
is injective. We will say that $\omega$ is pre-primitive for the origin $a$ if $\omega^{a}$ satisfies $(GC)$ et $(IC)$.
\end{definition}

\begin{remark} {\em (Analytization)} A Frobenius type structure $\fit$ on $\ait^{r}$ gives, after analytization, a Frobenius type structure
$$\fit^{an} =(\cit^{r}, E^{an}, \bigtriangledown^{an} , R_{0}^{an}, R_{\infty}^{an},\Phi^{an} ,g^{an})$$
on $\cit^{r}$. Notice that $E^{an}$ is canonically trivialized by a basis of (global) $\bigtriangledown$-flat sections.
A globally pre-primitive section $\omega$ of $E$ gives a pre-primitive section $\omega^{an}$ of $E^{an}$ for any choice of the origin.
\end{remark}

\subsection{Hertling and Manin's construction}

Let
$${\cal F}=(M, E, \bigtriangledown , R_{0}, R_{\infty},\Phi ,g)$$
be a Frobenius type structure on $M$, $\omega$ a $\bigtriangledown$-flat section of $E$ and $\chi_{\omega}$ the primitive map attached to $\omega$. If $\tilde{{\cal F}}$ is a deformation of ${\cal F}$, we will denote  $\tilde{\chi}_{\omega}$ ({\em resp.} $\tilde{\varphi}_{\omega}$) the primitive map ({\em resp.} the period map) attached to the flat extension of $\omega$. 
We will say that a $\bigtriangledown$-flat section of $E$ is {\em homogeneous} if it is an eigenvector of $R_{\infty}$.
Frobenius structures are defined in [Sab, VII.2].

\begin{theorem} Let $M$ be a germ of complex analytic manifold.\\
$1)$ {\em ([HeMa, theorem 2.5])} Assume that the Frobenius type structure ${\cal F}$
has a pre-primitive section $\omega$. Then ${\cal F}$ has a universal deformation. A deformation $\tilde{{\cal F}}$ of ${\cal F}$ 
is universal if and only if the primitive map (resp. period map) $\tilde{\chi}_{\omega}$ ( resp. $\tilde{\varphi}_{\omega}$) is a diffeomorphism (resp. an isomorphism).\\
$2)$ {\em ([HeMa, theorem 4.5])} A flat, pre-primitive and homogeneous section of the Frobenius type structure ${\cal F}$ defines, through the period map, a Frobenius structure on the base $\tilde{M}$ of any universal deformation of ${\cal F}$: $\tilde{M}$ is thus a Frobenius manifold.\\
$3)$ The Frobenius structures given by $2)$ on the bases of any two universal deformations are isomorphic. 
\end{theorem}

\begin{proof} 
$1)$ In brief, condition (GC) shows that one can construct deformations of the Frobenius type structure 
and condition (IC) is then used to show the universality of some of them: we will come back to this in section 5.2.\\ 
$2)$ It follows from $1)$ that ${\cal F}$ has a universal deformation 
$\tilde{{\cal F}}=(\tilde{M}, \tilde{E}, \tilde{\bigtriangledown} , \tilde{R}_{0}, \tilde{R}_{\infty}, \tilde{\Phi} , \tilde{g}).$
Moreover, the period map associated with the flat extension of the pre-primitive form is an isomorphism because the deformation is universal. One can thus carry the structures defined on $\tilde{E}$ onto $\Theta_{\tilde{M}}$, the sheaf of holomorphic vector fields on $\tilde{M}$, and gets, by definition, a ({\em a priori} non homogeneous) Frobenius structure on $\tilde{M}$. If moreover the pre-primitive form is homogeneous, its flat extension is also homogeneous because $R_{\infty}$ carries flat sections onto flat sections : this gives the homogeneity of the Frobenius structure. This shows that $\tilde{M}$ is a Frobenius manifold.\\
$3)$ Let $\tilde{{\cal F}}$ and $\tilde{{\cal F}}'$ be two universal deformations of ${\cal F}$, with bases $\tilde{M}$ and $\tilde{M}'$, $\tilde{\chi}_{\omega}$ ({\em resp.} $\tilde{\varphi}_{\omega}$)
and
$\tilde{\chi}_{\omega}'$ ({\em resp.} $\tilde{\varphi}_{\omega}'$) the respective primitive ({\em resp.} period) maps : these are diffeomorphisms ({\em resp.} isomorphisms). Write $\tilde{\chi}_{\omega}=\tilde{\chi}_{\omega}'\circ\psi$. Then $\tilde{\varphi}_{\omega}=\tilde{\varphi}_{\omega}'\circ T\psi$
where
$$T\psi: \Theta_{\tilde{M}}\rightarrow\Theta_{\tilde{M}'} $$
is the linear tangent map: it is an isomorphism which carries the structures from $\Theta_{\tilde{M}}$ onto $\Theta_{\tilde{M}'}$.
\end{proof}

\begin{example} Assume that $M=\{point\}$ and keep the notations of Remark 2.1.5 1). The Frobenius type structure
$(E, R_{0}, R_{\infty},g)$ has a universal deformation if $R_{0}$ is regular. This result was already known by B. Malgrange [Mal]. One gets a Frobenius structure on the base of any universal deformation of a regular Frobenius type structure if moreover $\omega$ is {\em homogeneous}. This is the setting of [DoSa2]. 
\end{example}

\section{Frobenius type structures and Laurent polynomials}

We explain here, and it is the first step, how to attach a Frobenius type structure on $\ait^{r}$ to  any convenient and nondegenerate Laurent polynomial.\\

\noindent {\em Until the end of this paper, $f$ will denote a convenient and nondegenerate Laurent polynomial, defined on the torus $U$}. 

\subsection{Subdiagram deformations} 
 If $f$ has a finite number of critical points, $\mu (f)$ will denote its global Milnor number, that is the sum of the Milnor numbers at its critical points. One attaches to $f$ its Newton polyhedron and an increasing filtration ${\cal N}_{\bullet}$ on $K$, indexed by $\qit$ and normalized such that $f\in {\cal N}_{1}K$ (see [K], we keep here the notations of [D]): this is the Newton filtration. This filtration induces a Newton filtration ${\cal N}_{\bullet}$ on $\Omega^{n}(U)$ such that $du/u\in {\cal N}_{0}\Omega^{n}(U)$.  Define
$${\cal N}_{<1}K:=\cup_{\alpha <1}{\cal N}_{\alpha}K,$$
which is a finite dimensional $\cit$-vector space, and $\nu:=\dim_{\cit}{\cal N}_{<1}K.$
Let
$$F:U\times \cit^{r}\rightarrow \cit$$
be the deformation of $f$ defined by  
$$F(u,x)=f(u)+\sum_{i=1}^{r}x_{i}g_{i}(u),$$
the $g_{i}$'s being Laurent polynomials. 

\begin{definition} $1)$ A Laurent polynomial $g$ is {\em subdiagram} if $g\in {\cal N}_{<1}K$.\\ 
$2)$ $F$ is a {\em subdiagram} deformation of $f$ if the Laurent polynomials $g_{i}$,  $i=1,\cdots ,r$, are subdiagram.\\
$3)$ The subdiagram deformation $F$ is {\em injective} if the $g_{i}$'s are $\cit$-linearly independent, {\em maximal} if it is injective and if  $r=\nu$ and {\em surjective} if $(g_{1},\cdots ,g_{r})$ is a lattice in $A_{f}$, {\em i.e} if every element in $A_{f}$ can be written as (the class of) a polynomial in 
$g_{1},\cdots ,g_{r}$ with coefficients in $\cit$.
\end{definition}

\begin{remark} Let $F_{1}^{max}$ and $F_{2}^{max}$ be two maximal subdiagram deformations. Then $F_{1}^{max}$ is surjective 
if and only if $F_{2}^{max}$ is so. In particular, if a maximal subdiagram deformation is surjective then any maximal subdiagram deformation will be so.
 \end{remark}

\subsection{The Brieskorn lattice of a subdiagram deformation}

 Let $F$ be a subdiagram deformation of $f$, $G_{0}$ ({\em resp.} $G$) its Brieskorn lattice ({\em resp.} its Gauss-Manin system),
$G_{0}^{o}$ ({\em resp.} $G^{o}$) the Brieskorn lattice ({\em resp.} the Gauss-Manin system) of $f$. One has
$$G_{0}^{o}=\frac{\Omega^{n}(U)[\theta ]}{(\theta d-df\wedge)\Omega^{n-1}(U)[\theta ]},$$
$$G_{0}=\frac{\Omega^{n}(U)[x,\theta ]}{(\theta d_{u}-d_{u}F\wedge)\Omega^{n-1}(U)[x,\theta ]}$$
where the notation $d_{u}$ means that the differential is taken with respect to $u$,
$$G=\frac{\Omega^{n}(U)[x,\theta ,\theta^{-1} ]}{(\theta d_{u}-d_{u}F\wedge)\Omega^{n-1}(U)[x,\theta ,\theta^{-1}]}$$
and
$$G^{o}=\frac{\Omega^{n}(U)[\theta ,\theta^{-1}]}{(\theta d-df\wedge)\Omega^{n-1}(U)[\theta ,\theta^{-1} ]}.$$
$G_{0}$ is a $\cit [x,\theta ]$-module and $G_{0}^{o}$ is a $\cit [\theta ]$-module. 
One defines a connection $\nabla$ on $G$ putting, for $\omega\in\Omega^{n}(U)[x]$, 
$$\theta^{2}\nabla_{\theta}(\omega\theta^{p})=F\omega\theta^{p}+p\omega\theta^{p+1}$$
and
$$\nabla_{\partial_{x_{j}}}(\omega\theta^{p})
=\partial_{x_{j}}(\omega )\theta^{p}-\frac{\partial F}{\partial x_{j}}\omega\theta^{p-1}.$$
Notice that these two operators commute with $\theta d_{u}-d_{u}F\wedge$ and that
$G_{0}$ is stable under $\theta^{2}\nabla_{\theta}.$
One defines in the same way the Brieskorn lattice $G_{0}^{a}$ and the Gauss-Manin system $G^{a}$ of $F_{a}:=F(.,a)$. 

Recall that the spectrum of $(G_{0}^{o},G^{o})$ is the set of the $\mu (f)$ rational numbers
$(\alpha_{1},\cdots ,\alpha_{\mu})$
such that
$$\sharp (i|\alpha_{i}=\alpha )=\dim_{\cit}\frac{ {\cal N}_{\alpha}\Omega^{n}(U)}{(df\wedge\Omega^{n-1}(U))\cap {\cal N}_{\alpha}\Omega^{n}(U)+{\cal N}_{<\alpha}\Omega^{n}(U)}.$$

\begin{theorem}
$1)$ $\mu (f)<+\infty$ and $G_{0}^{o}$ is a free $\cit [\theta ]$-module of rank $\mu (f)$.\\ 
$2)$ The Brieskorn lattice $G_{0}$ of any subdiagram deformation $F$ of $f$ is free, of rank $\mu (f)$, over $\cit [x,\theta ]$.\\
$3)$ Let $F$ be a subdiagram deformation of $f$. For any value $a$ of the parameter, one has
$\mu (F_{a})=\mu (f)$ and the spectrum of $(G_{0}^{a},G^{a})$ is equal to the one of $(G_{0}^{o},G^{o})$.

\end{theorem}
\begin{proof} From [K], one gets $\mu (f)<+\infty$ because $f$ is convenient and nondegenerate. The remaining assertions of $1)$ and $2)$ follow from the division theorem of Kouchnirenko, as stated in [DoSa, Lemma 4.6]: see [DoSa, Remark 4.8] for 1) and [D, Proposition 2.2.1] for 2). Let us show $3)$: if $f$ is convenient and nondegenerate, $F_{a}$ is so and the Newton polyhedra of $f$ and $F_{a}$ are the same : thus, the first assertion follows from [K]. If $\sum_{i}a_{i}u_{i}\frac{\partial f}{\partial u_{i}}\in {\cal N}_{\alpha}K$ one may assume, because of the division theorem quoted above, that $a_{i}\in {\cal N}_{\alpha -1}K$. Since the $g_{j}$'s are subdiagram, one gets $u_{i}\frac{\partial g_{j}}{\partial u_{i}}\in {\cal N}_{<1}K$. It follows that 
$$(df\wedge\Omega^{n-1}(U))\cap {\cal N}_{\alpha}+{\cal N}_{<\alpha}
=(dF_{a}\wedge\Omega^{n-1}(U))\cap {\cal N}_{\alpha}+{\cal N}_{<\alpha}.$$
This gives the second assertion.
\end{proof}

\subsection{The canonical Frobenius type structure of a subdiagram deformation}

Assume, and it is the starting point, that one has solved the Birkhoff problem for $G_{0}^{o}$, that is that one has found a basis $\varepsilon^{o}=(\varepsilon^{o}_{1},\cdots ,\varepsilon^{o}_{\mu})$ (we put here $\mu =\mu (f)$)
of $G_{0}^{o}$ over $\cit [\theta ]$, adapted to the microlocal Poincare duality (see [Sai], [DoSa2, p. 9] and also [D, paragraphe 3.3]), in which the matrix of the Gauss-Manin connection takes the form
$$-(\tau A_{0}^{o}+A_{\infty})\frac{d\tau }{\tau}$$
(we put $\tau :=\theta^{-1}$). This means that one can extend $G_{0}^{o}$ to a trivial bundle on $\ppit^{1}$ equipped with a meromorphic connection with logarithmic poles along $\tau =0$ and poles of order $1$ along $\tau =\infty$. One gets, using section 1.2, a Frobenius type structure 
$$(E^{o},R_{0}^{o},R_{\infty},g^{o})$$
on a point where\\

$\bullet$ $E^{o}=G_{0}^{o}/\theta G_{0}^{o}=\Omega^{n}(U)/df\wedge \Omega^{n-1}(U)$,\\

$\bullet$ $R_{0}^{o}$ ({\em resp.} $R_{\infty}$) is the endomorphism $E^{o}$ whose matrix is $A_{0}^{o}$
({\em resp.} $A_{\infty}$) in the basis of $E^{o}$ induced by $\varepsilon^{o}$.\\

\noindent It follows from section 3.2 that $R_{0}^{o}$ is the multiplication by $f$ on $E^{o}$.

 In this paper, we will always consider the canonical solution of the Birkhoff problem given by M. Saito's method  [Sai], [DoSa, Appendix B], [D1, section 6]. The endomorphism $R_{\infty}$ is in particular semi-simple and its eigenvalues run through the spectrum of  $(G_{0}^{o},G^{o})$. The basis $\varepsilon^{o}$ is homogeneous, that is
$R_{\infty}(\varepsilon^{o}_{i})=\alpha_{i}\varepsilon^{o}_{i}$
for all $i$, and we order $\varepsilon^{o}$ such that
$$\alpha_{1}\leq\cdots\leq\alpha_{\mu}.$$
Since $f$ is a convenient and nondegenerate Laurent polynomial, one has
$$\varepsilon^{o}_{1}=[\frac{du}{u}]$$
where $[\ ]$ denotes the class in $G_{0}^{o}$, $\alpha _{1}=0<\alpha_{2}$
(the 'multiplicity' of $\alpha_{1}$ in the spectrum is equal to $1$) and $\alpha_{\mu}=n>\alpha_{\mu -1}$ (see [DoSa, 4.d]).
 To any convenient and nondegenerate Laurent polynomial $f$, one attaches in this way a canonical Frobenius type structure on a point $(E^{o},R_{0}^{o},R_{\infty},g^{o}).$

\begin{theorem} 
Let $F$ be a subdiagram deformation of $f$ and 
$$E=G_{0}/\theta G_{0}=\Omega^{n}(U)[x]/d_{u}F\wedge\Omega^{n-1}(U)[x].$$
Then there exists a unique Frobenius type structure 
$$\fit_{o}=(\ait^{r}, E, \bigtriangledown , R_{0}, R_{\infty},\Phi ,g)$$
on $\ait^{r}$ such that
$$i^{*}_{\{0\}}\fit_{o}=(E^{o},R_{0}^{o},R_{\infty},g^{o}).$$
Moreover, for any value $a$ of the parameter, one has
$$i^{*}_{\{a\}}\fit_{o}=
(E^{a},R_{0}^{a},R_{\infty},g^{a}),$$
$(E^{a},R_{0}^{a},R_{\infty},g^{a})$ denoting the canonical Frobenius type structure attached to $F_{a}:=F(.,a)$.

\end{theorem}
\begin{proof} It follows from [D, Corollaire 3.1.3] that there exists a basis
$\varepsilon =(\varepsilon_{1},\cdots ,\varepsilon_{\mu})$
of $G_{0}$ over $\cit [x,\theta ]$ such that :\\
$1.$ the matrix of the connection $\nabla$ in this basis takes the form
$$-(\tau A_{0}(x)+A_{\infty})\frac{d\tau }{\tau}+\tau C(x)$$
where $C(x)=\sum_{i=1}^{r}C^{(i)}(x)dx_{i}$. 
The matrix $A_{0}(x)$ represents the multiplication by $F$ on $G_{0}/\tau^{-1} G_{0}$ in the basis induced by $\varepsilon$. Its coefficients belong to $\cit [x]$. 
The matrix $C^{(i)}$ represents the multiplication by $-g_{i}$ on $G_{0}/\tau^{-1} G_{0}$. Its coefficients belong also to $\cit [x]$. Last, the matrix
$A_{\infty}$ is constant.\\
$2.$ The restriction of $\varepsilon$ to the zero value of the parameters is equal to $\varepsilon^{o}$, the canonical solution of the Birkhoff problem for $G_{0}^{o}$.\\
The unicity of such a basis is classical (see [Mal] or [Sab, p. 209]).
Now one gets the desired Frobenius type structure $\fit_{o}$ using the results of section 1.2. The construction in [D] shows also that the restriction of the solution  $\varepsilon$ to any value $a$ of the parameter is the canonical solution of the Birkhoff problem for $G_{0}^{a}$. This gives the last assertion.
\end{proof}

\begin{definition} We will say that the Frobenius type structure $\fit_{o}$ constructed in Theorem 3.3.1 is the {\em canonical Frobenius type structure} attached to the subdiagram deformation $F$.
\end{definition}

\noindent In the notation $\fit_{o}$, the index $_{o}$ recalls the initial data (that is, $f$).

\subsection{Comparison of the canonical Frobenius type structures after a change of initial condition}

Let $F$ be a subdiagram deformation of $f$ and 
$(E^{a}, R_{0}^{a}, R_{\infty},g^{a})$
be the canonical Frobenius type structure on a point attached to $F_{a}=F(.,a)$.
 Let us also consider the subdiagram deformatiom of $F_{a}$ defined by
$$(u,x)\mapsto F(u,x+a).$$
By Theorem 3.3.1 there exists a unique Frobenius type structure on $\ait^{r}$ 
$$\fit_{a}=(\ait^{r}, E,\bigtriangledown , R_{0}, R_{\infty},\Phi ,g)$$
where
$$E:=\frac{\Omega^{n}(U)[x]}{d_{u}F(u,x+a)\wedge \Omega^{n-1}(U)[x]}$$
and such that
$$i^{*}_{\{0\}}\fit_{a}=(E^{a}, R_{0}^{a}, R_{\infty},g^{a}).$$

\noindent Let $\rho_{a}$ be the map defined by $\rho_{a}(x)=x+a$.

\begin{proposition} For any $a\in\ait^{r}$ one has $\fit_{a}=\rho_{a}^{*}\fit_{o}.$
\end{proposition}
\begin{proof} Follows from the unicity given by Theorem 3.3.1.
\end{proof}

\noindent This result says that the matrices attached by Lemma  1.1.2 to the Frobenius type structures involved are related by a translation: if
$B_{0}$ et $C^{(i)}$ ({\em resp.} $B_{0}'$ et $C^{(i)'}$) are the ones attached to $\fit_{o}$
({\em resp.} $\fit_{a}$) one has
$$B_{0}'(x)=B_{0}(x+a)$$
and
$$C^{(i)'}(x)=C^{(i)}(x+a).$$

\subsection{Comparison of the canonical Frobenius type structures attached to two different subdiagram deformations}  

We now compare the canonical Frobenius type structures attached to two different subdiagram deformations. 

\begin{proposition}
$1)$ Let $F^{max}$ and $G^{max}$ be two subdiagram maximal deformations of $f$, $\fit_{o}^{max}$ and $\git_{o}^{max}$ the canonical Frobenius type structures attached to $F^{max}$ and $G^{max}$ by Theorem 3.3.1. Then  $\fit_{o}^{max}$ and $\git_{o}^{max}$ are isomorphic.\\
$2)$ Let $\fit_{o}$ be the canonical Frobenius type structure attached to an injective subdiagram deformation $F$,  
 $\git_{o}^{max}$ the canonical Frobenius type structure attached to  a maximal subdiagram deformation $G^{max}$. Then $\fit_{o}$ is induced by $\git_{o}^{max}$ : there exists a map
$\Psi :\ait^{r}\rightarrow\ait^{\nu}$
such that $\fit_{o}=\Psi^{*}\git_{o}^{max}.$
\end{proposition}
\begin{proof} Write
$$F^{max}(u,x)=f(u)+\sum_{i=1}^{\nu}x_{i}g_{i}$$
and
$$G^{max}(u,x)=f(u)+\sum_{i=1}^{\nu}x_{i}g_{i}'.$$
Since $F^{max}$ and $G^{max}$ are maximal, $(g_{i})$ and $(g_{i}')$ are two basis of ${\cal N}_{<1}K$. In particular, there exists independent linear forms $L_{1},\cdots ,L_{\nu}$ such that
$$G^{max}(u,x)=f(u)+\sum_{i=1}^{\nu}L_{i}(x_{1},\cdots ,x_{\nu})g_{i}.$$ 
Define the map $\Phi$
by
$$\Phi (x_{1},\cdots ,x_{\nu})=(L_{1}(x_{1},\cdots ,x_{\nu}),\cdots ,L_{\nu}(x_{1},\cdots ,x_{\nu})).$$ 
Then $\git_{o}^{max}=\Phi^{*}\fit_{o}^{max}.$
This shows $1)$. $2)$ Follows from $1)$.
\end{proof}

\subsection{Good subdiagram deformations}

We define in this section a class of distinguished subdiagram deformations: these are the {\em good} subdiagram deformations. We will use these deformations in order to construct global deformations of the canonical Frobenius type structures along the subdiagram polynomials (see section 5.3).
If  $F$ is a subdiagram deformation of $f$, let, as in the proof of Theorem 3.3.1,
$\varepsilon =(\varepsilon_{1},\cdots ,\varepsilon_{\mu})$
be the canonical solution of the Birkhoff problem for the Brieskorn lattice $G_{0}$
of $F$. We order $\varepsilon$ such that
$$\alpha_{1}\leq\cdots \leq \alpha_{\mu},$$
the rational numbers $\alpha_{i}$ satisfying $R_{\infty}(\varepsilon_{i})=\alpha_{i}\varepsilon_{i}$.
Let $\fit_{o}$ be the canonical Frobenius type structure attached to $F$: we have a map
$$\Phi :E\rightarrow \Omega^{1}(\ait^{r})\otimes E.$$
Write $\Phi =\sum_{i}\Phi^{(i)}dx_{i}$. By definition, the $\Phi^{(i)}$'s are endomorphisms of $E$.

\begin{definition}
We will say that a subdiagram deformation $F$ is {\em good} if $F$ is injective and if
$$-\Phi^{(i)}(\varepsilon_{1} )=\varepsilon_{i}+\sum_{j<i}a^{j}_{i}(x)\varepsilon_{j}$$
for all $i$ ($a^{j}_{i}\in\cit [x]$). 
\end{definition}

\begin{proposition} There exists good ({\em resp.} good and maximal) subdiagram deformations. 
\end{proposition}

\noindent We will denote a good ({\em resp.} a good and maximal) subdiagram defomation by $F^{good}$ ({\em resp.} $F^{good,max}$).

\begin{proof} It is enough to work on the fiber above $0$: indeed, if $-\Phi^{(i)}(\varepsilon_{1}^{o})=\varepsilon_{i}^{o}$
for all $i$ one gets
$$-\Phi^{(i)}(\varepsilon_{1})=\varepsilon_{i}+\sum_{j<i}a^{j}_{i}(x)\varepsilon_{j}$$
because, the deformation being subdiagram, the principal parts are constant (see [D]).
Define, if 
$R_{\infty}(\varepsilon^{o}_{i})=\alpha_{i}\varepsilon^{o}_{i}$, 
$${\cal N}_{\alpha}(G_{0}^{o}\cap G_{\infty}^{o}):=\sum_{\alpha_{i}<\alpha }\cit \varepsilon^{o}_{i}.$$
By construction, one has (see [DoSa, appendix B] or [D1, paragraphe 6])
$$\frac{{\cal N}_{\alpha}(G_{0}^{o}\cap G_{\infty}^{o})}{{\cal N}_{<\alpha}(G_{0}^{o}\cap G_{\infty}^{o})}=
gr_{\alpha}^{{\cal N}}E^{o}$$
where $E^{o}=\Omega^{n}(U)/df\wedge\Omega^{n-1}(U)$ and ${\cal N}_{\bullet}$ is the Newton filtration induced on $E^{o}$.
If $\alpha <1$, it follows from [DoSa, Lemma 4.6] that
$$gr_{\alpha}^{{\cal N}}E^{o}=gr_{\alpha}^{{\cal N}}\Omega^{n}(U).$$
Since ${\cal N}_{<0}\Omega^{n}(U)={\cal N}_{<0}(G_{0}^{o}\cap G_{\infty}^{o})=0$, one deduces that
$${\cal N}_{\alpha}\Omega^{n}(U)={\cal N}_{\alpha}(G_{0}^{o}\cap G_{\infty}^{o})$$
for all $\alpha <1$. This shows two things : first that, if $R_{\infty}(\varepsilon^{o}_{i})=\alpha_{i}\varepsilon^{o}_{i}$,
one has $\alpha_{i}<1$ for all $i\in \{1,\cdots ,\nu\}$ and second, that, given $\varepsilon^{o}_{i}$
such that $\alpha_{i}<1$, there exists a unique subdiagram Laurent polynomial $g_{i}$ such that
$$[g_{i}\frac{du}{u}]=\varepsilon_{i}^{o}.$$
To simplify, put $\varepsilon_{i}^{o}=g_{i}$. Then, for $r\leq\nu$, 
$$F^{good}(u,x)=f(u)+\sum_{i=1}^{r}x_{i}\varepsilon^{o}_{i}$$
is clearly injective and is a good subdiagram deformation.
The subdiagram deformation
$$F^{good,max}(u,x)=f(u)+\sum_{i=1}^{\nu}x_{i}\varepsilon^{o}_{i}$$
is good and maximal.
\end{proof}

\noindent Let $\fit_{o}^{good}$ ({\em resp.} $\fit_{o}^{good,max}$) be the Frobenius type structure attached to a good
({\em resp.} to a good and maximal) subdiagram deformation $F^{good}$ ({\em resp.} $F^{good,max}$).

\begin{lemma} 
$1)$ Assume that $(g_{1},\cdots ,g_{r})$ is a lattice in $A_{f}$. Then $F^{good,max}$ is surjective.\\
$2)$ $\fit_{o}^{good,max}$ is isomorphic to any canonical Frobenius type structure attached to a maximal subdiagram deformation and it induces any canonical Frobenius type structure attached to an injective subdiagram deformation.
\end{lemma}
\begin{proof} 
$1)$ Follows from Remark 3.1.2 and $2)$ follows from Proposition 3.5.1 because $F^{good,max}$ is a maximal subdiagram deformation. 
\end{proof}

\section{Pre-primitive forms of a canonical Frobenius type structure}

Let $f$ be a convenient and nondegenerate Laurent polynomial,  
$$F(u,x)=f(u)+\sum_{i=1}^{r}x_{i}g_{i}$$ 
be a subdiagram deformation of $f$ and 
$$\fit_{o}=(\ait^{r}, E, \bigtriangledown , R_{0}, R_{\infty},\Phi ,g)$$
the canonical Frobenius type structure on  $\ait^{r}$ attached to $F$.

\subsection{The form $\omega$}
Let
$\varepsilon =(\varepsilon_{1},\cdots ,\varepsilon_{\mu})$ be the (ordered) solution of the Birkhoff problem for $G_{0}$ considered in section 3.6.

\begin{proposition} One has
$$\varepsilon_{1}=[\frac{du}{u}]$$
where $[\ ]$ denotes the class in $G_{0}$. In particular, the class of
$\frac{du}{u}$ in $E$ is $\bigtriangledown$-flat
and homogeneous, {\em i.e} an eigenvector of $R_{\infty}$.
 \end{proposition}
\begin{proof} 
Let $V_{\bullet}$ be the Malgrange-Kashiwara filtration along $\tau =0$ of the Gauss-Manin system $G$ of the subdiagram deformation $F$, $V_{\bullet}G_{0}$ its trace on $G_{0}$. In the convenient and nondegenerate case, this filtration is equal to
the Newton filtration ${\cal N}$ (up to a shift) [D, proposition 2.3.3].
It follows from [D, proposition 2.3.1] that $V_{\alpha_{1}}G_{0}$ is a free $\cit [x]$-module and, from [D1, proposition 7.0.2], that every basis of $V_{\alpha_{1}}G_{0}$ is a part of a solution of the Birkhoff problem for $G_{0}$.
Now, $V_{\alpha_{1}}G_{0}$ is of rank $1$ over $\cit [x]$
(for all $a$ the $\cit$-vector space $V_{\alpha_{1}}G_{0}^{a}$
is $1$-dimensional), $\varepsilon_{1}$ is a basis of it and $\frac{du}{u}\in V_{\alpha_{1}}G_{0}$. 
Notice that, because $F_{a}$ is a convenient and non degenerate Laurent polynomial,
$\varepsilon_{1}^{a}$ is equal to the 
class of the form $du/u$ in $G_{0}^{a}$ for all $a$ [DoSa, 4.d]. If 
$[du/u]=p(x)\varepsilon_{1}$ in $G_{0}$, we deduce from this that $p$ is identically equal to $1$.
\end{proof}

\begin{notation} Until the end of this paper, $\omega$ will denote the class of $\frac{du}{u}$
in $E.$
\end{notation}

\subsection{Conditions (IC) and (GC) for $\omega^{o}$}

 Choose an origin, say $0$. We have 
$$ E^{o}=E/(x)E=\Omega^{n}(U)/df\wedge\Omega^{n-1}(U)$$ 
\noindent and $\omega^{o}$ denotes the class of $\frac{du}{u}$
in $E^{o}$. 
Conditions (IC) and (GC) for $\omega^{o}$ are defined in 2.1.7.

\begin{lemma}
$1)$ $\omega^{o}$ satisfies (IC) if and only if the classes of $g_{1},\cdots ,g_{r}$ are linearly independent in $A_{f}$.\\
$2)$ $\omega^{o}$ satisfies (GC) if and only if every element of $A_{f}$ can be written as (the class of) a polynomial in $g_{1},\cdots ,g_{r},f$ with coefficients in $\cit$.
\end{lemma}
\begin{proof} By definition (see section 3.2), one has 
 $R_{0}^{o}(\omega^{o})=[f\frac{du}{u}]$ and
$-\Phi_{\omega}^{o}(\partial_{x_{i}})=[g_{i}\frac{du}{u}]$
where $[\ ]$ denotes the class in $E^{o}$. 
\end{proof}

\noindent The following proposition justifies Definition  3.1.1 3):

\begin{proposition} 
$1)$ If the deformation $F$ is injective then $\omega^{o}$ satisfies (IC).\\
$2)$ If the deformation $F$ is surjective then $\omega^{o}$ satisfies (GC).
\end{proposition}
\begin{proof}  
Let us show $1)$ : it follows from Lemma 4.2.1 1) that it is enough to show that the classes of $g_{1},\cdots ,g_{r}$ in $A_{f}$ are linearly independent. But this follows from the conditions $g_{j}\in {\cal N}_{\alpha_{j}}K$ with $\alpha_{j}<1$ for all $j$: indeed, assume that there exist complex numbers $\alpha_{1},\cdots ,\alpha_{r}$ such that 
$$\sum_{j=1}^{r}\alpha_{j}g_{j}=\sum_{i=1}^{n}b_{i}u_{i}\frac{\partial f}{\partial u_{i}}.$$
One can choose, using [DoSa, lemma 4.6], the $b_{i}$'s such that  $b_{i}\in {\cal N}_{\alpha -1}K$ where $\alpha :=max_{j}\alpha_{j}$. We then get $b_{i}=0$ for all $i$ because $\alpha <1$. Moreover, the $g_{j}$' are linearly independant in $K$ (the deformation $F$ is injective): this shows that $\alpha_{i}=0$ for all $i$. $2)$ is clear.
\end{proof}

\begin{example} We will say that the subdiagram deformation $F$ contains the monomial $u_{1}^{a_{1}}\cdots u_{n}^{a_{n}}$ if there exists $j$ such that $g_{j}(u)= u_{1}^{a_{1}}\cdots u_{n}^{a_{n}}$. Assume that the injective deformation $F$ contains 
the monomials $u_{1},\cdots ,u_{n},u_{1}^{-1},\cdots ,u_{n}^{-1}.$
Then $\omega^{o}$ satisfies $(IC)$ and $(GC)$. Notice that, often, the monomials $1/u_{1},\cdots ,1/u_{n}$ are equal, in $A_{f}$, to a (positive) power of the monomials 
$u_{1},\cdots ,u_{n}$: in this case, the condition  ``{\em $F$ contains the monomials $u_{1},\cdots ,u_{n}$}'' is enough to get the condition (GC) for $\omega^{o}$. 
\end{example}

\begin{lemma} Assume that the deformation $F$ is injective. Then $\omega^{a}$ satisfies (IC) for any choice of origin $a$.
\end{lemma}
\begin{proof} It is enough to show that the classes of $g_{1},\cdots ,g_{r}$ in $A_{F_{a}}$ are linearly independent. But one can repeat the proof of Proposition 4.2.2, because $F_{a}$ is convenient and non degenerate and because the Newton polyhedra (and hence the Newton filtrations) of $f$ and $F_{a}$ are the same.
\end{proof}

\subsection{The canonical pre-primitive form}

Le $\fit_{o}^{an}$ be the analytization of the Frobenius type structure $\fit_{o}$ (see Remark 2.1.8), $\fit_{o,0}^{an}$ its germ at $0$.

\begin{proposition}
$1)$ Assume that the subdiagram deformation $F$ is injective and surjective. Then $\omega^{an}$ is a pre-primitive section of $\fit_{o,0}^{an}$.\\
$2)$ Assume that the subdiagram deformation $F$ is injective. Then $\omega$ satisfies $(IC)^{gl}$. If moreover $F$ contains the monomials 
$u_{1},\cdots ,u_{n},u_{1}^{-1},\cdots ,u_{n}^{-1}$
then $\omega$ is a globally pre-primitive section of the Frobenius type structure $\fit_{o}$ and $\omega^{an}$ is a pre-primitive section of the Frobenius type structure $\fit_{o}^{an}$ for any choice of the origin in $\cit^{r}$.
\end{proposition}
\begin{proof}
$1)$ follows from proposition 4.2.2. 
  A section of the kernel of the period map $\varphi_{\omega}$ is given by a finite number of polynomials that vanishes every where by Lemma 4.2.4. This shows the first assertion of $2)$. With the supplementary given assumption, $\omega$ satisfies of course $(GC)^{gl}$. The results about $\omega^{an}$ are then clear.
\end{proof}

\section{Deformations and universal deformations of the canonical Frobenius type structure} 

We keep here the situation and the notations of section 4.

\subsection{Deformations of the canonical Frobenius type structure}

Let $C(x)$, $B_{0}(x)$ and $B_{\infty}$ be the matrices attached to $\fit_{o}$ by Lemma 1.1.2. 
Recall the conditions $(GC)^{gl}$ and $(IC)^{gl}$ for $\omega$, given in Definition 2.1.7.

\begin{lemma} Assume that $\omega$ satisfies $(GC)^{gl}$. Let  
$f_{11},\cdots ,f_{\mu 1}$ be elements of $\cit [x]\{y\}$ (resp. ${\cal O}(\cit^{r})\{y\}$), $y\in\cit$, 
such that $f_{i1}(x,0)=0$ for $i=1,\cdots ,\mu$. Then there exists a unique t-uple of matrices
$$(C(x,y),B_{0}(x,y),B_{\infty})$$
such that\\
$1)$ the coefficients of $C(x,y)$ and $B_{0}(x,y)$ belong to $\cit [x]\{y\}$ (resp. ${\cal O}(\cit^{r})\{y\}$),\\
$2)$ $C(x,0)=C(x)$, $B_{0}(x,0)=B_{0}(x)$ and $\frac{\partial f_{i1}}{\partial y}(x,y)=D_{i1}(x,y)$ if
$$C(x,y)=\sum_{i=1}^{r}C^{(i)}(x,y)dx_{i}+D(x,y)dy,$$ 
$3)$ the relations of Lemma 1.1.2 are satisfied.
\end{lemma}
\begin{proof} See [HeMa, Theorem 2.5]. It remains to show that the coefficients of  
$C(x,y)$ and $B_{0}(x,y)$ belong to $\cit [x]\{y\}$ (resp. ${\cal O}(\cit^{r})\{y\}$), but this follows from the fact that the coefficients of $C(x)$ and $B_{0}(x)$ belong to $\cit [x]$ (by Theorem 3.3.1) and from condition $(GC)^{gl}$.
\end{proof}

\begin{example} Assume that $f_{11}(x,y)=y$ and $f_{i1}(x,y)=0$ for $i=2,\cdots ,\mu$. Lemma 1.1.2 gives 
$$C^{(i)}_{j1}(x,y)=C^{(i)}_{j1}(x)$$ 
for all $i$ and for all $j$, $D_{11}(x,y)=1$
and $D_{j1}(x,y)=0$ if $j\neq 1$. 
\end{example}

\noindent By induction, one shows that Lemma 5.1.1 remains true if $y=(y_{1},\cdots ,y_{\ell})\in\cit^{\ell}.$

\begin{corollary} Assume that $\omega$ satisfies $(GC)^{gl}$. Then,\\
$1)$ for any choice of functions
$$f_{11},\cdots ,f_{\mu 1}\in {\cal O}(\cit^{r})\{y_{1},\cdots ,y_{\ell }\}$$
such that $f_{i1}(x,0)=0$ there exists a {\em unique} deformation 
$$\tilde{\fit}_{o}^{an}=(\cit^{r}\times (\cit^{\ell},0) , \tilde{E}, \tilde{\bigtriangledown} , \tilde{R}_{0}, \tilde{R}_{\infty},\tilde{\Phi}, \tilde{g})$$
on $\cit^{r}\times (\cit^{\ell},0)$ of the canonical Frobenius type structure $\fit_{o}^{an}$.\\
\noindent $2)$ Any deformation 
$$\tilde{\fit}_{o}^{'an}=(\cit^{r}\times (\cit^{\ell '},0) , \tilde{E}, \tilde{\bigtriangledown} , \tilde{R}_{0}, \tilde{R}_{\infty}, \tilde{\Phi}, \tilde{g})$$ 
$\cit^{r}\times (\cit^{\ell '},0)$ of $\fit_{o}^{an}$ can be obtained as in 1). 
\end{corollary}
\begin{proof} For $1)$, it remains to show the assertion on the metric $\tilde{g}$ :  $\tilde{g}$ is the unique $\tilde{\bigtriangledown}$-flat metric on  $\tilde{E}$ extending $g$. Starting with a basis adapted to $g$, and keeping the notations of Lemma 5.1.1, it suffices to show that if the initial data are symmetric, then the matrices 
$C(x,y)$ and $B_{0}(x,y)$ are so : one can argue by induction as in the proof of Lemma 5.1.1 (see [KS, corollary 1.17], [HeMa, lemma 3.2] and also [D, paragraphe 3.3]). Let us show $2)$ : if  
\begin{eqnarray}
\tilde{\chi}_{\omega^{an}}' : & \cit^{r}\times (\cit^{\ell '},0) & \rightarrow  \cit^{\mu} \\
                        & (x,y)                   & \mapsto   (\Gamma_{11}(x,y),\cdots ,\Gamma_{\mu 1}(x,y))
\end{eqnarray}
is the primitive map attached to the deformation $\tilde{\fit}_{o}^{'an}$ and to the flat extension of $\omega^{an}$, one puts $f_{i1}(x,y)=\Gamma_{i1}(x,y)-\Gamma_{i1}(x,0)$. 
\end{proof}

\begin{remark} Assume that $\omega^{o}$ satisfies (GC) for the origin $0$. One gets in the same way deformations of  
$\fit_{o,0}^{an}:=\cit \{x\}\otimes \fit_{o}.$
The functions $f_{i1}$ now belong to $\cit \{x,y\}$ and the coefficients of the matrices involved are holomorphic. This is the setting of [HeMa]. 
\end{remark}

 Let $a\in\cit^{r}$ and $\rho_{a}$  be the map defined by $\rho_{a}(x,y)=(x+a,y)$. 

\begin{corollary} Assume that $\omega$ satisfies condition $(GC)^{gl}$. Let $\tilde{\fit}_{o}^{an}$ be the deformation of $\fit_{o}^{an}$ given by Corollary 5.1.3 for a choice of functions $f_{i1}$. 
Then  $\rho^{*}_{a}\tilde{\fit}_{o}^{an}$ is the deformation of the Frobenius type structure $\fit_{a}^{an}$ given by Corollary 5.1.3 for the functions $f_{i1}\circ\rho_{a}$. 
\end{corollary}
\begin{proof}
 Follows from  Proposition 3.4.1.
\end{proof}

\subsection{Local universal deformations of the canonical Frobenius type structure}

We consider in this section only germs : we fix an origin, say $0$, and we work with 
$\fit^{an}_{o,0}$, the germ of $\fit_{o}^{an}$ at $0$.
Universal deformations of $\fit^{an}_{o,0}$ are defined in section 1.3.

\begin{lemma} Assume that $\omega^{o}$ satisfies (GC). Then $\fit^{an}_{o,0}$ has a universal deformation.
\end{lemma}
\begin{proof} Since $\omega^{o}$ satisfies the condition (GC), we can start with a deformation $\tilde{\fit}^{an}_{o,0}$
of $\fit^{an}_{o,0}$ given by Remark 5.1.4. Let 

\begin{eqnarray}
\tilde{\chi}_{\omega^{an}} : & (\cit^{r}\times \cit^{\ell},0) & \rightarrow  (\cit^{\mu},0) \\
                      & (x,y)                   & \mapsto   (\Gamma_{11}(x,y),\cdots ,\Gamma_{\mu 1}(x,y))  
\end{eqnarray}

\noindent be the primitive map attached to $\tilde{\fit}^{an}_{o,0}$ and to the flat extension of $\omega^{an}$. The Frobenius type structure
$$\tilde{\fit}^{an}_{o,0}=((\cit^{r}\times \cit^{\ell},0) , \tilde{E}, \tilde{\bigtriangledown} , \tilde{R}_{0}, \tilde{R}_{\infty}, \tilde{\Phi}, \tilde{g})$$
is a universal deformation of the canonical Frobenius type structure 
$$\fit^{an}_{o,0}=((\cit^{r},0), E, \bigtriangledown , R_{0}, R_{\infty},\Phi ,g)$$
if and only if $\tilde{\chi}_{\omega^{an}}$ is a diffeomorphism: this is precisely what gives [HeMa, p. 123].
Now, $\omega^{o}$ satisfies also (IC): one can choose the $f_{i1}$'s such that $\tilde{\chi}_{\omega^{an}}$ is (at least locally) invertible. We get in this way a universal deformation of $\fit^{an}_{o,0}$.
\end{proof}

\begin{corollary} Let $F$ be an injective subdiagram deformation of $f$. Then,\\ 
$1)$ the Frobenius type structure $\fit^{an}_{o,0}$ has a universal deformation, if $\omega^{o}$ satisfy (GC).\\ 
$2)$ Assume that $F$ contains the monomials 
$u_{1},\cdots ,u_{n},u_{1}^{-1},\cdots ,u_{n}^{-1}.$
Then the Frobenius type structure $\fit_{o}^{an}$ has a universal deformation in the neighborhood of any point of $\cit^{r}$.
\end{corollary}
\begin{proof} $1)$ Since the deformation $F$ is injective, $\omega^{o}$ satisfies also (IC) because of Lemma 4.2.2 1): $\omega^{an}$ is thus pre-primitive for the origin $0$. $2)$ Follows from Proposition 4.3.1.
\end{proof}

\subsection{Semi-global universal deformations of the canonical Frobenius type structure}

We globalize here the results of section 5.2 along $\cit^{r}$ ({\em i.e} along the subdiagram monomials).
We give first the analog of definition 1.3.1 3): 

\begin{definition} Let $\tilde{\fit}_{o}^{an}$ be a deformation of $\fit_{o}^{an}$ on $\cit^{r}\times (\cit^{\ell},0)$ as in Corollary 5.1.3. We say that  $\tilde{\fit}_{o}^{an}$ is a {\em semi-global} universal deformation of $\fit_{o}^{an}$ if, for any other deformation $\tilde{\fit}_{o}^{'an}$ 
on $\cit^{r}\times (\cit^{\ell '},0)$ of $\fit_{o}^{an}$, there exists a unique map 
 $$\Psi :\cit^{r}\times (\cit^{\ell '},0)\rightarrow \cit^{r}\times (\cit^{\ell },0),$$
inducing the identity on $\cit^{r}$, such that $\Psi^{*}\tilde{\fit}_{o}^{an}=\tilde{\fit}_{o}^{'an}.$
\end{definition}

We show first that such semi-global universal deformations exist if $F$ is a good subdiagram deformation (Definition 3.6.1). The following lemma is an analog of Lemma 5.2.1: 

\begin{lemma} Let $\fit_{o}^{good}$ be the canonical Frobenius type structure attached to a good subdiagram deformation of $f$. Then :\\
$1)$ the primitive map $\chi_{\omega}$ attached to $\fit_{o}^{good}$ takes the form
$$\chi_{\omega}(x_{1},\cdots ,x_{r})=(-x_{1}+G_{1}(x_{2},\cdots ,x_{r}),-x_{2}+G_{2}(x_{3},\cdots ,x_{r}),\cdots ,-x_{r-1}+G_{r-1}(x_{r}),-x_{r},0,\cdots,0)$$
where $G_{1}$, $G_{2}$,..., $G_{r-1}$ are suitable polynomial functions.\\
$2)$ Assume moreover that $\omega$ satisfies $(GC)^{gl}$. Choose $f_{i1}(x,y)=0$ for $i=1,\cdots ,r$, $f_{i1}(x,y)=y_{i-r}$ for $i=r+1,\cdots ,\mu$ and let $\tilde{\fit}_{o}^{good, an}$ be the deformation of $\fit_{o}^{good, an}$ given by the corollary 5.1.3. Its primitive map 
$$\tilde{\chi}_{\omega^{an}} : \cit^{r}\times (\cit^{\mu -r},0)  \rightarrow  \cit^{r}\times (\cit^{\mu -r},0)$$
takes the form
$$\tilde{\chi}_{\omega^{an}}(x_{1},\cdots ,x_{r},y_{1},\cdots ,y_{\mu -r}) =(-x_{1}+G_{1}(x_{2},\cdots ,x_{r}) ,\cdots ,-x_{r-1}+G_{r-1}(x_{r}),-x_{r},y_{1},\cdots ,y_{\mu -r})$$
and $\tilde{\fit}_{o}^{good, an}$ is a semi-global universal deformation of $\fit_{o}^{good, an}$.
\end{lemma}
\begin{proof}
$1)$  By definition of the good subdiagram deformations, we have, in $G_{0}$, 
$$-\Phi^{(i)}(\omega )=\varepsilon_{i}+\sum_{j<i}a^{j}_{i}(x)\varepsilon_{j}$$
for all $i$, with $a^{j}_{i}\in\cit [x]$.
Let $\Gamma_{j1}$ be such that
$$d\Gamma_{j1}(x)=\sum_{i=1}^{r}C^{(i)}_{j1}(x)dx_{i}$$
with the initial data $\Gamma_{j1}(0)=0$. One has $d\Gamma_{j1}(x)=0$ for $j>r$ hence $\Gamma_{j1}(x)=0$ for $j>r$. 
In the same way, one gets $d\Gamma_{r1}(x)=-dx_{r}$
and
$$d\Gamma_{j1}(x)=-dx_{j}+\sum_{i=j+1}^{r}C^{(i)}_{j1}(x)dx_{i}$$
for $j=1,\cdots ,r-1$. The result follows. Now $2)$ follows from $1)$ and Example 5.1.2. One gets the universality as in the proof of Lemma 5.2.1 (notice that $\tilde{\chi}_{\omega^{an}}^{-1}$ is also polynomial in $x$).
\end{proof}

\begin{corollary} 
The canonical Frobenius type structures attached to a {\em maximal} subdiagram deformation have semi-global universal deformations, if $\omega$ satisfies $(GC)^{gl}$.
\end{corollary}
\begin{proof} Apply Lemma 5.3.2 to a good and maximal subdiagram deformation (it exists by Proposition 3.6.2) and use Lemma  3.6.3 2).
\end{proof}

\noindent Finally, using Corollary 5.1.5, we get

\begin{corollary} Let $F$ be a maximal subdiagram deformation of $f$ and assume that $\omega$ satisfies $(GC)^{gl}$ (for $F$). For any $a\in\cit^{r}$, the Frobenius type structure  $\fit_{a}^{an}$ has a semi-global universal deformation of $\tilde{\fit}_{a}^{an}$ satisfying
$$\tilde{\fit}_{a}^{an}=\rho_{a}^{*}\tilde{\fit}_{o}^{an}.$$
\end{corollary}

\section{Application: construction of Frobenius manifolds}

 Let $f$ be a convenient and nondegenerate Laurent polynomial, $\mu$ its global Milnor number,
$$F(u,x)=f(u)+\sum_{i=1}^{r}x_{i}g_{i}(u)$$ 
be a subdiagram deformation of $f$, 
$$\fit_{o}=(\ait^{r}, E, \bigtriangledown , R_{0}, R_{\infty},\Phi ,g)$$
be the canonical Frobenius type structure attached to $F$ by Theorem 3.3.1 and $\fit_{o}^{an}$ its analytization (see remark 2.1.8). Let $\omega$ be the class of  $\frac{du}{u}$
in $E$.

\subsection{Local setting}

We work in this section with punctual germs. Let
$\fit^{an}_{o,0}$ be the germ of $\fit_{o}^{an}$ at $0$.
 The following theorems show that one can equip $(\cit^{\mu},0)$ with a canonical Frobenius structure: $(\cit^{\mu},0)$ is thus a Frobenius manifold.

\begin{theorem} Assume that the subdiagram deformation $F$ is injective and surjective. Then :\\
$1)$ $\omega^{an}$ is a $\bigtriangledown$-flat and homogeneous section of $E^{an}$.\\
$2)$ $\omega^{an}$ is pre-primitive for the origin $0$.\\
$3)$ $\fit^{an}_{o,0}$ has a universal deformation $\tilde{\fit}^{an}_{o,0}$.\\
$4)$ The pre-primitive section $\omega^{an}$ defines a Frobenius structure on the base of the universal deformation $\tilde{\fit}^{an}_{o,0}$. The Frobenius structures obtained in this way on the bases of any two universal deformations are isomorphic.
\end{theorem}
\begin{proof} 
$1)$ follows from Proposition 4.1.1,
$2)$ from proposition 4.2.2 and 
$3)$ from $2)$ and Lemma 5.2.1. Last,
$4)$ follows from Theorem 2.2.1 2) et 3).
\end{proof}

 Let us show now that the Frobenius structures constructed in Theorem 6.1.1 do not depend on the choice of the subdiagram deformations.

\begin{lemma} Let $F$ (resp. $G$) be an injective and surjective subdiagram deformations of $f$, $\fit_{o}$ (resp. $\git_{o}$) be the canonical Frobenius type structure attached to $F$ (resp. $G$). If they exist, the universal deformations
$\tilde{\fit}^{an}_{o,0}$ of $\fit^{an}_{o,0}$ and 
$\tilde{\git}^{an}_{o,0}$ of $\git^{an}_{o,0}$ are isomorphic. 
\end{lemma}
\begin{proof} 
Extend $F$ ({\em resp.} $G$) to a maximal deformation $F^{max}$ ({\em resp.} $G^{max}$): this is always possible because $F$ and $G$ are injective.
It follows from Proposition 3.5.1 that the respective canonical Frobenius structures $\fit^{max}_{o}$ and $\git^{max}_{o}$ are isomorphic. 
Thus, since $\fit^{max,an}_{o,0}$ is a deformation $\fit^{an}_{o,0}$, $\tilde{\fit}^{an}_{o,0}$
is a universal deformation of $\fit^{max,an}_{o,0}$ and also a universal deformation of $\git^{max,an}_{o,0}$. Because $\tilde{\git}^{an}_{o,0}$ is a universal deformation of $\git^{max,an}_{o,0}$, we deduce that $\tilde{\fit}^{an}_{o,0}$ and 
$\tilde{\git}^{an}_{o,0}$ are isomorphic.
\end{proof}

\begin{theorem} Let $F$ and $G$ be two injective and surjective subdiagram deformations of $f$, $\fit_{o}$ (resp. $\git_{o}$) be the canonical Frobenius type structure attached to $F$ (resp. $G$). Then:\\
$1)$ $\omega^{an}$ is a $\bigtriangledown$-flat and homogeneous section of the bundles associated with $\fit_{o}$ and $\git_{o}$.\\
$2)$ $\omega^{an}$ is pre-primitive for the origin $0$.\\
$3)$ $\fit^{an}_{o}$ (resp. $\git^{an}_{o}$) has a universal deformation $\tilde{\fit}^{an}_{o}$ (resp. 
$\tilde{\git}^{an}_{o}$). $\tilde{\fit}^{an}_{o}$ and $\tilde{\git}^{an}_{o}$ are isomorphic.\\
$4)$ The Frobenius structures defined by the pre-primitive form $\omega^{an}$ according to Theorem 6.1.1 do not depend, up to isomorphism, on the choice of the  subdiagram deformations $F$ and $G$.
\end{theorem}
\begin{proof} Because of Theorem 6.1.1, it is enough to show $3)$ and $4)$: $3)$ follows from Lemma 6.1.2 and $4)$ is then clear (see Theorem 2.2.1). 
\end{proof}

\noindent This shows Theorem 1 in the introduction.

\subsection{Globalization}

 Recall that $\fit_{a}$ denotes the canonical Frobenius type structure attached to $F_{a}=F(.,a)$ and that $\rho_{a}$ is the map defined by $\rho_{a}(x,y)=(x+a,y)$ for $(x,y)\in\cit^{r}\times (\cit^{\ell},0)$. Good subdiagram deformations are defined in section 3.6.

\begin{theorem} Assume that the subdiagram deformation $F$ is injective. Then\\
$1)$ $\omega$ is a $\bigtriangledown$-flat, homogeneous section of $E$.\\
Assume moreover that $\omega$ satisfies $(GC)^{gl}$. Then :\\
$2)$ $\fit_{o}^{an}$ has deformations on $\cit^{r}\times (\cit^{\ell}, 0)$.\\ 
$3)$ If $F$ is a maximal subdiagram deformation then $\fit_{o}^{an}$ has a semi-global universal deformation $\tilde{\fit}_{o}^{an}$ and, for any 
$a\in\cit^{r}$, $\fit_{a}^{an}$ has a semi-global universal deformation $\tilde{\fit}_{a}^{an}$ satisfying
 $$\tilde{\fit}_{a}^{an}=\rho^{*}_{a}\tilde{\fit}_{o}^{an}.$$
The period map $\tilde{\varphi}_{\omega^{an}}$ (resp. $\tilde{\varphi}_{\omega^{an}} '$) attached to 
$\tilde{\fit}_{o}^{an}$ (resp. $\tilde{\fit}_{a}^{an}$) defines a Frobenius structure on $\cit^{r}\times (\cit^{\mu -r},0)$ and
$\tilde{\varphi}_{\omega^{an}}'=\rho_{a}\circ\tilde{\varphi}_{\omega^{an}}$.

\end{theorem}
\begin{proof} 
1) Follows from Proposition 4.1.1,
2) from Corollary 5.1.3 and 3) from Corollary 5.3.3, Corollary 5.3.4 and 
Theorem 2.2.1.
\end{proof}

\begin{corollary} Assume that $\omega$ satisfies $(GC)^{gl}$ for $F$.\\
$1)$ The Frobenius structures on $(\cit^{\mu},0)$, attached by 
Theorem 6.1.1 to the convenient and nondegenerate Laurent polynomials $F_{a}$
are isomorphic to the pull-back by $\rho_{a}$ of the one attached to $f$.\\
$2)$ For {\em any} injective subdiagram deformation $G$ of $f$, the Frobenius structures on $(\cit^{\mu},0)$, attached by 
Theorem 6.1.1 to the convenient and nondegenerate Laurent polynomials $G_{a}$
are isomorphic to the pull-back by $\rho_{a}$ of the one attached to $f$.\\
\end{corollary}
\begin{proof}  Because of Theorem 6.1.3 4) one may assume that $F$ is a maximal subdiagram deformation and thus apply Theorem 6.2.1 3) to get $1)$. Let us show 2) : let $G$ be any injective subdiagram deformation of $f$. Without loss of generality, one may assume that $G$ is maximal. 
It follows from proposition 3.5.1 that the canonical Frobenius type structures attached to $G$ and $F$ (say, $\git_{o}$ and $\fit_{o}$)
satisfy $\git_{o}=\Phi^{*}\fit_{o}$ where $\Phi$ is an isomorphism.
Thus, for any $a\in\cit^{r}$, $\git_{a}=\Psi^{*}\fit_{a}$
where $\Psi$ is also an isomorphism by Proposition 3.4.1 and 2) follows from 1).
\end{proof}

\noindent This shows Theorem 2 in the introduction.

\end{document}